\documentclass[11pt]{article}
\usepackage{amsfonts,amscd,amssymb, amsmath}
\usepackage[usenames]{color}
\usepackage{graphicx}
\usepackage{euscript}
\usepackage[bookmarksnumbered,plainpages,driverfallback=dvipdfm,backref=page]{hyperref}
\usepackage{enumerate}
\usepackage[all,tips]{xy}
\usepackage{epic,eepic}
\usepackage{color}
\usepackage{tikz}

\usepackage[hyperpageref]{backref}
\backref
\renewcommand*{\backref}[1]{}  
\renewcommand*{\backrefalt}[4]{
  \ifcase #1 %
  \relax
  \or
(Cited page~#2.)%
  \else
(Cited pages~#2.)%
  \fi}

\newtheorem{definition}{Definition}[section]
\newtheorem{lemma}[definition]{Lemma}
\newtheorem{proposition}[definition]{Proposition}
\newtheorem{corollary}[definition]{Corollary}
\newtheorem{remark}[definition]{Remark}

\newtheorem{theorem}[definition]{Theorem}
\newtheorem{example}[definition]{Example}

\def\rawo\lonra{\longrightarrow}

\allowdisplaybreaks[4]
\newenvironment{proof}{{\it Proof.}}{\hfill $ \square $ \vskip 4mm}

\begin{document}

\title{BiHom-NS-algebras, twisted Rota-Baxter operators and generalized Nijenhuis operators
}
\author{Ling Liu \\
College of Mathematics and Computer Science,\\
Zhejiang Normal University, 
Jinhua 321004, China \\
e-mail: ntliulin@zjnu.cn \and Abdenacer Makhlouf \\
Universit\'{e} de Haute Alsace, \\
IRIMAS-D\'epartement  de Math\'{e}matiques,  \\
18, rue des fr\`{e}res Lumi\`{e}re, F-68093 Mulhouse, France\\
e-mail: Abdenacer.Makhlouf@uha.fr \and Claudia Menini \\
University of Ferrara, \\
Department of Mathematics and Computer Science\\
Via Machiavelli 30, Ferrara, I-44121, Italy\\
e-mail: men@unife.it \and Florin Panaite \\
Institute of Mathematics of the Romanian Academy,\\
PO-Box 1-764, RO-014700 Bucharest, Romania\\
e-mail: florin.panaite@imar.ro }
\date{}
\maketitle

\begin{abstract}
The purpose of this paper is to introduce and study BiHom-NS-algebras, which are a generalization of  NS-algebras using two homomorphisms. Moreover, we discuss their relationships with twisted Rota-Baxter operators in a BiHom-associative context. Furthermore, we introduce a generalization of Nijenhuis operators that lead to BiHom-NS-algebras along BiHom-associative algebras.\\

\begin{small}
\end{small}
\end{abstract}

\section*{Introduction}

NS-algebras (corresponding to associative algebras) have been introduced by Leroux (\cite{leroux}) and independently by Uchino 
(\cite{uchinoarXiv}, which is an earlier, arXiv version of \cite{uchino}), as algebras with three operations $\prec $, $\succ $ and 
$\vee $ satisfying certain axioms that imply that the new operation $\ast =\prec +\succ +\vee $ is associative. NS-algebras 
generalize both dendriform (\cite{loday}) and tridendriform (\cite{lodayronco}) algebras. Examples are obtained via so-called twisted Rota-Baxter operators (see \cite{uchino}), which are a generalization of $\mathcal{O}$-operators involving a Hochschild 2-cocycle,  and via Nijenhuis operators (see \cite{leroux}). We recall from \cite{carinena} that a Nijenhuis operator $N:A\rightarrow A$ 
on an associative algebra $(A, \mu )$ with multiplication denoted by $\mu (x\otimes y)=xy$, for $x, y\in A$, is a linear map 
satisfying  
\begin{eqnarray}
&&N(x)N(y)=N(N(x)y+xN(y)-N(xy)), \;\;\; \forall \; x, y\in A. \label{Nijen}
\end{eqnarray}
By \cite{leroux}, if one defines $x\prec y=xN(y)$, $x\succ y=N(x)y$ and $x\vee y=-N(xy)$, then $(A, \prec , \succ , \vee )$ 
is an NS-algebra, and in particular the new multiplication defined on $A$ by $x\ast y=xN(y)+N(x)y-N(xy)$ is associative. 
Basic examples (see \cite{carinena}) are obtained by taking a fixed element $a\in A$ and defining $N_1, N_2:A\rightarrow A$ 
by $N_1(x)=ax$ and $N_2(x)=xa$, for all $x\in A$; it turns out that $N_1, N_2$ are Nijenhuis operators and in each case 
the new multiplication $\ast $ as above boils down to $x\ast y=xay$, for all $x, y\in A$. This property can be regarded also in the following (converse) way: the fact that the new multiplication on $A$ given by $x\ast y=xay$ is associative (usually, this 
new operation $\ast $ is said to be a ''perturbation'' of the old multiplication of $A$ via the element $a$)
can be obtained as a consequence of a property of Nijenhuis operators (or, alternatively, that it can be given a Nijenhuis operator 
interpretation). 

Let us mention that one can define NS-algebras corresponding to other classes of algebras than associative, for instance corresponding to Lie or Leibniz algebras (see \cite{daslie}, \cite{dasleibniz}) and, much more generally, to any class of algebras defined by multilinear relations (see \cite{opv}). 

Hom-type and BiHom-type algebras are certain algebraic structures (of growing interest in recent years) whose study began in some early papers such as \cite{JDS}, \cite{DS}, \cite{ms} and more recently \cite{gmmp}, and can be roughly described as being defined by some identities obtained by twisting a classical algebraic identity (such as associativity) by one or two maps. For instance, 
a BiHom-associative algebra $(A, \mu , \alpha , \beta )$ is an algebra $(A, \mu )$, with notation $\mu :A\otimes A\rightarrow A$, 
$\mu (x\otimes y)=xy$, together with two (multiplicative with respect to $\mu $)
commuting linear maps (called structure maps) $\alpha , \beta :A\rightarrow A$ 
such that $\alpha (x)(yz)=(xy)\beta (z)$ for all $x, y, z\in A$. There exist BiHom analogues of many types of algebras, in particular 
of (tri)dendriform algebras, infinitesimal bialgebras etc (see for instance \cite{lmmp1}, \cite{lmmp2}, \cite{lmmp3}, \cite{usjgeomphys} and references therein). 
Examples of (Bi)Hom-type algebras can be obtained from classical types of algebras by a procedure called ''Yau twisting''. 

The BiHom analogue of the ''perturbations'' mentioned above has been introduced in \cite{usjgeomphys} as follows. 
Let $(A, \mu , \alpha , \beta )$ be a BiHom-associative algebra and let $a\in A$ be such that $\alpha ^2(a)=\beta ^2(a)=a$. Define a new operation on $A$ by $x\ast y=\alpha (x)(\alpha (a)y)$; then $(A, \ast , \alpha ^2, \beta ^2)$ is a BiHom-associative algebra.

The starting point of this paper was to look for a ''Nijenhuis operator interpretation'' of this fact. One can notice that 
Nijenhuis operators defined by the relation (\ref{Nijen}) can be considered on any type of algebra (not necessarily associative), so we were trying to find a Nijenhuis operator on $(A, \mu , \alpha , \beta )$ depending on the given element $a$ and that would lead to the operation $\ast $, but we failed.  It turns out (just as it happened before with a certain context in which one was forced to consider a generalized version of Rota-Baxter operators, see \cite{canadian}) that the solution to this problem was to consider a generalized version of Nijenhuis operators on BiHom-associative algebras (defined by the axioms (\ref{extracom1})-(\ref{genNijsup2}) below). And indeed, the operators $N_1, N_2:A\rightarrow A$, $N_1(x)=\alpha (a)x$ and 
$N_2(x)=x\alpha (a)$, for $x\in A$, are such generalized Nijenhuis operators from which one can obtain the multiplication $\ast $ in a certain way. 

We were then led to introduce the concept of BiHom-NS-algebra, the BiHom analogue of Leroux's and Uchino's NS-algebras. 
They generalize BiHom-(tri)dendriform algebras, and it turns out that the generalized Nijenhuis operators that we introduced lead to BiHom-NS-algebras. We define as well the BiHom analogue of twisted Rota-Baxter operators and prove that they also lead to BiHom-NS-algebras and that, moreover, just as in the classical case in \cite{uchino}, there is an adjunction between BiHom-NS-algebras and twisted Rota-Baxter operators. 

\section{Preliminaries}\label{sec0}
\setcounter{equation}{0}

We work over a base field $\Bbbk $. All
algebras, linear spaces etc. will be over $\Bbbk $; unadorned $\otimes $
means $\otimes_{\Bbbk}$. Unless otherwise specified, the
algebras that will appear in what follows are
\emph{not} supposed to be associative or 
unital, and the multiplication $\mu :A\otimes
A\rightarrow A$ of an algebra $(A, \mu )$ is denoted by $\mu
(a\otimes a^{\prime })=aa^{\prime }$. 
For the composition of two maps $f$
and $g$, we will write either $g\circ f$ or simply $gf$. For the identity
map on a linear space $V$ we will use the notation $id_V$.

\begin{definition} (\cite{leroux})
\label{Def:NS-alg}An NS-algebra $\left( A,\prec ,\succ ,\vee \right) $
is a $\Bbbk$-linear space $A$ equipped with linear maps $\prec ,\succ ,\vee
:A\otimes A\rightarrow A$ satisfying the following relations (for all $%
x,y,z\in A$):%
\begin{eqnarray}
&&\left( x\prec y\right)  \prec z=x\prec \left( y\ast z\right) ,  \label{NS1}
\\
&&\left( x\succ y\right)  \prec z=x\succ \left( y\prec z\right) ,\quad
\label{NS2} \\
&&\left( x\ast y\right)  \succ z=x\succ \left( y\succ z\right),   \label{NS3}
\\
&&\left( x\vee y\right)  \prec z+\left( x\ast y\right) \vee z=x\succ \left(
y\vee z\right) +x\vee \left( y\ast z\right),   \label{NS4}
\end{eqnarray}%
where%
\begin{eqnarray}
&&x\ast y=x\prec y+x\succ y+x\vee y.  \label{NS*}
\end{eqnarray}
A morphism $f:(A,\prec ,\succ ,\vee )\rightarrow (A^{\prime },\prec ^{\prime
},\succ ^{\prime },\vee ^{\prime })$ of NS-algebras is a linear map $%
f:A\rightarrow A^{\prime }$ satisfying $f(x\prec y)=f(x)\prec ^{\prime }f(y)$%
, $f(x\succ y)=f(x)\succ ^{\prime }f(y)$ and $f(x\vee y)=f(x)\vee ^{\prime
}f(y)$, for all $x,y\in A$. An immediate consequence is that
\begin{eqnarray}
f\left( x\ast y\right) =f\left( x\right) \ast ^{^{\prime }}f\left( y\right), \;\;\; \forall \; x, y\in A.
\label{fstar}
\end{eqnarray}
\end{definition}
\begin{definition}
(\cite{gmmp}) A BiHom-associative algebra is a
4-tuple $\left( A,\mu ,\alpha ,\beta \right) $, where $A$ is a $\Bbbk $%
-linear space, $\alpha :A\rightarrow A$, $\beta :A\rightarrow A$ and $\mu
:A\otimes A\rightarrow A$ are linear maps, with notation $\mu (x\otimes y)
=xy$, for all $x, y\in A$, satisfying the following conditions, for all $x,
y, z\in A$:
\begin{gather}
\alpha \circ \beta =\beta \circ \alpha , \\
\alpha (xy) =\alpha (x)\alpha (y) \text{ and }\beta (xy)=\beta (x)\beta (y)
,\quad \text{(multiplicativity)}  \label{eqalfabeta} \\
\alpha (x)(yz)=(xy)\beta (z).\quad \text{(BiHom-associativity)}
\label{eqasso}
\end{gather}

The maps $\alpha $ and $\beta $ (in this order) are called the structure maps of $A$.

A morphism $f:(A,\mu _{A},\alpha _{A},\beta _{A})\rightarrow (B,\mu
_{B},\alpha _{B},\beta _{B})$ of BiHom-associative algebras is a linear map $%
f:A\rightarrow B$ such that $\alpha _{B}\circ f=f\circ \alpha _{A}$, $\beta
_{B}\circ f=f\circ \beta _{A}$ and $f\circ \mu _{A}=\mu _{B}\circ (f\otimes
f)$.
\end{definition}
\begin{definition}(\cite{lmmp1}) \label{bimodule}
Let $(A, \mu _A , \alpha _A, \beta _A)$ be a BiHom-associative algebra and $(M, \alpha _M, \beta _M)$ a triple where $M$
is a $\Bbbk$-linear space and $\alpha _M, \beta _M:M \rightarrow M$ are commuting linear maps.\\
(i) $(M, l, \alpha _M, \beta _M)$ is a left $A$-module if
we have a linear map $l:A\otimes M\rightarrow M$, $x\otimes m\mapsto x\cdot m$, such that
$\alpha _M(x\cdot m)=\alpha _A(x)\cdot \alpha _M(m)$,
$\beta _M(x\cdot m)=\beta _A(x)\cdot \beta _M(m)$ and
\begin{eqnarray}
&&\alpha _A(x)\cdot (x^{\prime }\cdot m)=(xx^{\prime })\cdot \beta _M(m),\;\;\;\;\forall \;\;x, x'\in A, \;m\in M.
\label{lmod}
\end{eqnarray}
(ii) $(M, r, \alpha _M, \beta _M)$ is a right $A$-module if
we have a linear map $r:M\otimes A\rightarrow M$, $m\otimes x\mapsto m\cdot x$, for which  
$\alpha _M(m\cdot x)=\alpha _M(m)\cdot \alpha _A(x)$,
$\beta _M(m\cdot x)=\beta _M(m)\cdot \beta _A(x)$ and
\begin{eqnarray}
&&\alpha _M(m)\cdot (xx')=(m\cdot x)\cdot \beta _A(x'), \;\;\;\;\forall \;\;x, x'\in A, \;m\in M.
\label{rmod}
\end{eqnarray}
(iii) If $(M, l, \alpha _M, \beta _M)$ is
a left $A$-module and $(M, r, \alpha _M, \beta _M)$ is a right $A$-module, with notation as above, 
then $(M, l, r, \alpha _M, \beta _M)$
 is called an $A$-bimodule if
\begin{eqnarray}
&&\alpha _A(x)\cdot (m\cdot x')=(x\cdot m)\cdot \beta _A(x'), \;\;\;\;\forall \;\;x, x'\in A, \;m\in M. \label{bimod}
\end{eqnarray}
\end{definition}
\begin{definition} (\cite{lmmp1})
A BiHom-tridendriform algebra is a 6-tuple $(A,\prec ,\succ
,\cdot ,\alpha ,\beta )$, where $A$ is a $\Bbbk$-linear space and $\prec ,\succ
,\cdot :A\otimes A\rightarrow A$ and $\alpha ,\beta :A\rightarrow A$ are
linear maps such that:
\begin{eqnarray}
&&\alpha \circ \beta =\beta \circ \alpha ,  \label{BiHomtridend1} \\
&&\alpha (x\prec y)=\alpha (x)\prec \alpha (y),~\alpha (x\succ y)=\alpha
(x)\succ \alpha (y),~\alpha (x\cdot y)=\alpha (x)\cdot \alpha (y),
\label{BiHomtridend4} \\
&&\beta (x\prec y)=\beta (x)\prec \beta (y),~\beta (x\succ y)=\beta (x)\succ
\beta (y),~\beta (x\cdot y)=\beta (x)\cdot \beta (y),  \label{BiHomtridend5}
\\
&&(x\prec y)\prec \beta (z)=\alpha (x)\prec (y\prec z+y\succ z+y\cdot z),
\label{BiHomtridend8} \\
&&(x\succ y)\prec \beta (z)=\alpha (x)\succ (y\prec z),
\label{BiHomtridend9} \\
&&\alpha (x)\succ (y\succ z)=(x\prec y+x\succ y+x\cdot y)\succ \beta (z),
\label{BiHomtridend10} \\
&&\alpha (x)\cdot (y\succ z)=(x\prec y)\cdot \beta (z),
\label{BiHomtridend11} \\
&&\alpha (x)\succ (y\cdot z)=(x\succ y)\cdot \beta (z),
\label{BiHomtridend12} \\
&&\alpha (x)\cdot (y\prec z)=(x\cdot y)\prec \beta (z),
\label{BiHomtridend13} \\
&&\alpha (x)\cdot (y\cdot z)=(x\cdot y)\cdot \beta (z),
\label{BiHomtridend14}
\end{eqnarray}%
for all $x,y,z\in A$. The maps $\alpha $ and $\beta $ (in this order) are called the
structure maps of $A$.

A BiHom-tridendriform algebra $(A,\prec ,\succ
,\cdot ,\alpha ,\beta )$ for which $x\cdot y=0$ for all $x, y\in A$ is called a BiHom-dendriform algebra. 
\end{definition}
\section{BiHom-NS-algebras}\label{sec1}
\setcounter{equation}{0}
We begin by introducing the BiHom analogue of classical NS-algebras. 
\begin{definition}\label{DefBHNS}
A BiHom-NS-algebra is a 6-tuple $\left( A,\prec ,\succ ,\vee ,\alpha ,\beta
\right) $ consisting of a $\Bbbk$-linear space $A$ equipped with linear maps $\prec
,\succ ,\vee :A\otimes A\rightarrow A$ and $\alpha, \beta :A\rightarrow A$
satisfying the following conditions (for all $x, y, z\in A$):
\begin{eqnarray}
&&\alpha \circ \beta  =\beta \circ \alpha ,   \label{BiHomNS0} \\
&&\alpha \left( x\prec y\right)  =\alpha \left( x\right) \prec \alpha \left(
y\right) ,\alpha \left( x\succ y\right) =\alpha \left( x\right) \succ \alpha
\left( y\right) ,\alpha \left( x\vee y\right) =\alpha \left( x\right) \vee
\alpha \left( y\right),   \label{BiHomNS1} \\
&&\beta \left( x\prec y\right)  =\beta \left( x\right) \prec \beta \left(
y\right) ,\beta \left( x\succ y\right) =\beta \left( x\right) \succ \beta
\left( y\right) ,\beta \left( x\vee y\right) =\beta \left( x\right) \vee
\beta \left( y\right),   \label{BiHomNS2} \\
&&\left( x\prec y\right)  \prec \beta \left( z\right) =\alpha \left(
x\right) \prec \left( y\ast z\right) ,  \label{BiHomNS3} \\
&&\left( x\succ y\right)  \prec \beta \left( z\right) =\alpha \left(
x\right) \succ \left( y\prec z\right) ,\quad   \label{BiHomNS4} \\
&&\left( x\ast y\right)  \succ \beta \left( z\right) =\alpha \left( x\right)
\succ \left( y\succ z\right),   \label{BiHomNS5} \\
&&\left( x\vee y\right)  \prec \beta \left( z\right) +\left( x\ast y\right)
\vee \beta \left( z\right) =\alpha \left( x\right) \succ \left( y\vee
z\right) +\alpha \left( x\right) \vee \left( y\ast z\right),
\label{BiHomNS6}
\end{eqnarray}
where
\begin{eqnarray}
&&x\ast y=x\prec y+x\succ y+x\vee y.  \label{BiHomNS*}
\end{eqnarray}

The maps $\alpha $ and $\beta $ (in this order) are called the
structure maps of $A$.

A morphism $f:(A,\prec ,\succ ,\vee , \alpha , \beta )\rightarrow (A^{\prime },\prec ^{\prime
},\succ ^{\prime },\vee ^{\prime }, \alpha ^{\prime}, \beta ^{\prime})$ of BiHom-NS-algebras is a linear map $
f:A\rightarrow A^{\prime }$ satisfying $\alpha ^{\prime}\circ f=f\circ \alpha $, $\beta ^{\prime}\circ f=f\circ \beta $,
$f(x\prec y)=f(x)\prec ^{\prime }f(y)$%
, $f(x\succ y)=f(x)\succ ^{\prime }f(y)$ and $f(x\vee y)=f(x)\vee ^{\prime
}f(y)$, for all $x,y\in A$.
\end{definition}

Similarly to the characterization of BiHom-dendriform algebras in terms of bimodules from \cite{lmmp1}, Proposition 3.16, 
we characterize BiHom-NS-algebras as follows.
\begin{proposition}\label{BHNStoBHA}
Let $\left( A,\prec ,\succ ,\vee ,\alpha ,\beta \right) $ be a 6-tuple where $A$ is a 
$\Bbbk$-linear space and $\prec
,\succ ,\vee :A\otimes A\rightarrow A$ and $\alpha, \beta :A\rightarrow A$ are linear maps such that $\alpha $ and $\beta $ are 
multiplicative with respect to $\prec $, $\succ $ and $\vee $. Define a new multiplication on $A$ by 
$x\ast y=x\prec y+x\succ y+x\vee y$, for all $x, y\in A$. Then 
$\left( A,\prec ,\succ ,\vee , \alpha , \beta \right) $ is a  BiHom-NS-algebra if and only if 
$\left( A,\ast ,\alpha ,\beta \right) $ is a BiHom-associative algebra, denoted by $A_{bhas}$, and 
$(A, \succ , \prec, \alpha , \beta )$ is an $A_{bhas}$-bimodule (notation for bimodules as in Definition \ref{bimodule}). 
\end{proposition}

\begin{proof} Assume first that $\left( A,\prec ,\succ ,\vee , \alpha , \beta \right) $ is a  BiHom-NS-algebra. 
It is easy to see that $\alpha $ and $\beta $ are multiplicative with respect to $\ast $, so we only check
$\left( \ref{eqasso}\right) $. We compute, for all $x,y, z \in A$: \\[2mm]
${\;\;\;\;\;}$
$\alpha (x)\ast (y\ast z)$
\begin{eqnarray*}
&\overset{\left( \ref{BiHomNS*}\right) }{=}&\alpha
\left( x\right)  \succ \left( y\prec z+y\succ z+y\vee z\right)
 +\alpha \left( x\right)  \prec \left( y\ast z\right) +\alpha \left( x\right) \vee \left( y\ast z\right)  \\
&=& \alpha \left( x\right) \succ \left( y\prec z\right) +\alpha \left(
x\right) \succ \left( y\succ z\right) +\alpha \left( x\right) \succ \left(
y\vee z\right)\\
 &&+\alpha \left( x\right)  \prec \left( y\ast z\right) + \alpha \left( x\right) \vee (y\ast z) \\
&\overset{\left( \ref{BiHomNS3}\right) ,\left( \ref{BiHomNS4}\right) ,\left( %
\ref{BiHomNS5}\right) }{=}&\left( x\ast y\right)  \succ \beta \left(
z\right) +\left( x\succ y\right) \prec \beta \left( z\right) +\left[ \alpha
\left( x\right) \succ \left( y\vee z\right) +\alpha \left( x\right) \vee
(y\ast z)\right]  \\
&&+\left( x\prec y\right)  \prec \beta \left( z\right)  \\
&\overset{\left(\ref{BiHomNS6}\right)}{=}&\left( x\ast y\right)  \succ \beta \left(
z\right) +\left( x\succ y\right) \prec \beta \left( z\right) +\left[ \left(
x\vee y\right) \prec \beta \left( z\right) +\left( x\ast y\right) \vee \beta
\left( z\right) \right]  \\
&&+\left( x\prec y\right)  \prec \beta \left( z\right) \\
&=&\left( x\ast
y\right) \succ \beta \left( z\right) +\left( x\prec y +
x\succ y +x\vee y \right) \prec \beta \left( z\right)
\\
&&+ \left( x\ast y\right) \vee \beta
\left( z\right)\\
&\overset{\left( \ref{BiHomNS*}\right) }{=}&\left( x\ast y\right)  \prec
\beta \left( z\right) +\left( x\ast y\right) \succ \beta \left( z\right)
+\left( x\ast y\right) \vee \beta \left( z\right) =\left( x\ast y\right)
\ast \beta \left( z\right),
\end{eqnarray*}
proving that $(A, \ast , \alpha , \beta )$ is BiHom-associative. The fact that $(A, \succ , \prec, \alpha , \beta )$ is an $A_{bhas}$-bimodule follows immediately from (\ref{BiHomNS3}),  (\ref{BiHomNS4}) and (\ref{BiHomNS5}). 

Conversely, (\ref{BiHomNS3}),  (\ref{BiHomNS4}) and (\ref{BiHomNS5}) follow from the fact that 
$(A, \succ , \prec, \alpha , \beta )$ is an $A_{bhas}$-bimodule, while (\ref{BiHomNS6}) is obtained by subtracting 
(\ref{BiHomNS3}),  (\ref{BiHomNS4}) and (\ref{BiHomNS5}) from the BiHom-associativity condition 
$\alpha (x)\ast (y\ast z)=(x\ast y)\ast \beta (z)$ written explicitely for $\ast =\prec +\succ +\vee$. 
\end{proof}

BiHom-NS-algebras can be obtained by Yau twisting from classical NS-algebras as follows.
\begin{proposition}
\label{YauNS} Let $(A,\prec ,\succ ,\vee )$ be an NS-algebra and $\alpha
,\beta :A\rightarrow A$ two commuting NS-algebra morphisms. Define $%
\prec _{(\alpha ,\beta )},\succ _{(\alpha ,\beta )}, \vee _{(\alpha , \beta )}:A\otimes A\rightarrow A$
by
\begin{equation*}
x\prec _{(\alpha ,\beta )}y=\alpha (x)\prec \beta (y), \;\;\;\;\;\;\;x\succ
_{(\alpha ,\beta )}y=\alpha (x)\succ \beta (y), \;\;\;\;\;\;\;x\vee
_{(\alpha ,\beta )}y=\alpha \left( x\right) \vee \beta \left( y\right), 
\end{equation*}%
for all $x,y\in A$. Then $A_{(\alpha ,\beta )}:=(A,\prec _{(\alpha ,\beta
)},\succ _{(\alpha ,\beta )},\vee _{(\alpha ,\beta )},\alpha ,\beta )$ is a
BiHom-NS-algebra, called the Yau twist of $A$. Moreover, assume that $%
(A^{\prime },\prec ^{\prime },\succ ^{\prime },\vee ^{\prime })$ is another
NS-algebra and $\alpha ^{\prime },\beta ^{\prime }:A^{\prime }\rightarrow
A^{\prime }$ are two commuting NS-algebra morphisms and $f:A\rightarrow
A^{\prime }$ is a morphism of NS-algebras satisfying $f\circ \alpha =\alpha
^{\prime }\circ f$ and $f\circ \beta =\beta ^{\prime }\circ f$. Then $%
f:A_{(\alpha ,\beta )}\rightarrow A_{(\alpha ^{\prime },\beta ^{\prime
})}^{\prime }$ is a morphism of BiHom-NS-algebras.
\end{proposition}

\begin{proof} One can easily see that, if we define $\ast _{(\alpha , \beta )}:=\prec _{(\alpha ,\beta
)}+\succ _{(\alpha ,\beta )}+\vee _{(\alpha ,\beta )}$, then
\begin{equation}
x\ast _{(\alpha ,\beta )}y=\alpha \left( x\right) \ast \beta \left( y\right), \;\;\; \forall \;\;x, y\in A.
\label{staralphabeta}
\end{equation}%
The conditions  $\left( \ref{BiHomNS1}\right)$, $\left( \ref{BiHomNS2}\right)$ are easy to prove and  left to the reader.
We verify $\left( \ref{BiHomNS3}\right)$, $\left( \ref{BiHomNS4}\right)$ and $\left( \ref{BiHomNS5}\right)$:
\begin{eqnarray*}
\left( x\prec _{(\alpha ,\beta )}y\right)  \prec_{(\alpha ,\beta )} \beta
\left( z\right) &=&\alpha \left( \alpha \left( x\right) \prec \beta \left(
y\right) \right) \prec \beta ^{2}\left( z\right) =\left( \alpha ^{2}\left(
x\right) \prec \alpha \beta \left( y\right) \right) \prec \beta ^{2}\left(
z\right)  \\
&\overset{\left( \ref{NS1}\right) }{=}&\alpha ^{2}\left( x\right)  \prec
\left( \alpha \beta \left( y\right) \ast \beta ^{2}\left( z\right) \right)
=\alpha ^{2}\left( x\right) \prec \left( \beta \alpha \left( y\right) \ast
\beta ^{2}\left( z\right) \right) \\
&\overset{\left( \ref{fstar}\right) }{=}&
\alpha ^{2}\left( x\right) \prec \beta \left( \alpha \left( y\right) \ast
\beta \left( z\right) \right)  \\
&\overset{\left( \ref{staralphabeta}\right) }{=}&\alpha ^{2}\left( x\right)
\prec \beta \left( y\ast _{(\alpha ,\beta )}z\right) =\alpha \left(
x\right) \prec _{(\alpha ,\beta )}\left( y\ast _{(\alpha ,\beta )}z\right),
\end{eqnarray*}
\begin{eqnarray*}
\left( x\succ _{(\alpha ,\beta )}y\right)  \prec_{(\alpha ,\beta )} \beta
\left( z\right) &=&\alpha \left( x\succ _{(\alpha ,\beta )}y\right) \prec
\beta ^{2}\left( z\right) =\alpha \left( \alpha \left( x\right) \succ \beta
\left( y\right) \right) \prec \beta ^{2}\left( z\right)  \\
&=&\left( \alpha ^{2}\left( x\right) \succ \alpha \beta \left( y\right)
\right) \prec \beta ^{2}\left( z\right) =\left( \alpha ^{2}\left( x\right)
\succ \beta \alpha \left( y\right) \right) \prec \beta ^{2}\left( z\right)\\
&\overset{\left( \ref{NS2}\right) }{=} &
\alpha ^{2}\left( x\right)  \succ \beta \left( \alpha \left(
y\right) \prec \beta \left( z\right) \right)  \\
&=&\alpha ^{2}\left(
x\right) \succ \beta \left( y\prec _{(\alpha ,\beta )}z\right) =\alpha
\left( x\right) \succ _{(\alpha ,\beta )} \left( y\prec _{(\alpha ,\beta
)}z\right) ,
\end{eqnarray*}%
\begin{eqnarray*}
\left( x\ast _{(\alpha ,\beta )}y\right)  \succ_{(\alpha ,\beta )} \beta
\left( z\right) &=&\alpha \left( x\ast _{(\alpha ,\beta )}y\right) \succ \beta
^{2}\left( z\right) \overset{\left( \ref{staralphabeta}\right) }{=}\alpha
\left( \alpha \left( x\right) \ast \beta \left( y\right) \right) \succ \beta
^{2}\left( z\right)  \\
&\overset{\left( \ref{fstar}\right) }{=}&\left( \alpha ^{2}\left( x\right)
\ast \alpha \beta \left( y\right) \right)  \succ \beta ^{2}\left( z\right)
=\left( \alpha ^{2}\left( x\right) \ast \beta \alpha \left( y\right) \right)
\succ \beta ^{2}\left( z\right) \\
&\overset{\left( \ref{NS3}\right) }{=}&\alpha
^{2}\left( x\right) \succ \left( \beta \alpha \left( y\right) \succ \beta
^{2}\left( z\right) \right)
=\alpha ^{2}\left( x\right) \succ \beta \left( \alpha \left( y\right)
\succ \beta \left( z\right) \right) \\
&=&\alpha ^{2}\left( x\right) \succ \beta
\left( y\succ _{(\alpha ,\beta )}z\right) =\alpha \left( x\right) \succ
_{(\alpha ,\beta )}\left( y\succ _{(\alpha ,\beta )}z\right) .
\end{eqnarray*}%
Now we check the condition $\left( \ref{BiHomNS6}\right)$:\\[2mm]
${\;\;\;\;\;}$
$\left( x\vee _{(\alpha ,\beta )}y\right)  \prec_{(\alpha ,\beta )} \beta
\left( z\right) +\left( x\ast _{(\alpha ,\beta )}y\right) \vee _{(\alpha
,\beta )}\beta \left( z\right)$
\begin{eqnarray*}
&=&\alpha \left( x\vee _{(\alpha ,\beta
)}y\right) \prec \beta ^{2}\left( z\right) +\alpha \left( x\ast _{(\alpha
,\beta )}y\right) \vee \beta ^{2}\left( z\right)  \\
&\overset{\left( \ref{staralphabeta}\right) }{=}&\alpha \left( \alpha \left(
x\right) \vee \beta \left( y\right) \right)  \prec \beta ^{2}\left(
z\right) +\alpha \left( \alpha \left( x\right) \ast \beta \left( y\right)
\right) \vee \beta ^{2}\left( z\right)  \\
&\overset{\left( \ref{fstar}\right) }{=}&\left( \alpha ^{2}\left( x\right) \vee \alpha \beta \left( y\right)
\right) \prec \beta ^{2}\left( z\right) +\left( \alpha ^{2}\left( x\right)
\ast \alpha \beta \left( y\right) \right) \vee \beta ^{2} \left( z\right)  \\
&\overset{\left( \ref{NS4}\right) }{=}&\alpha ^{2}\left( x\right)  \succ
\left( \alpha \beta \left( y\right) \vee \beta ^{2}\left( z\right) \right)
+\alpha ^{2}\left( x\right) \vee \left( \alpha \beta \left( y\right) \ast
\beta ^{2}\left( z\right) \right)  \\
&\overset{\left( \ref{fstar}\right) }{=}&\alpha ^{2}\left( x\right) \succ \left( \beta \alpha \left( y\right) \vee
\beta ^{2}\left( z\right) \right) +\alpha ^{2}\left( x\right) \vee \beta
\left( \alpha \left( y\right) \ast \beta \left( z\right) \right)  \\
&=&\alpha ^{2}\left( x\right) \succ \beta \left( \alpha \left( y\right) \vee
\beta \left( z\right) \right) +\alpha ^{2}\left( x\right) \vee \beta \left(
\alpha \left( y\right) \ast \beta \left( z\right) \right)  \\
&\overset{\left( \ref{staralphabeta}\right) }{=}&\alpha ^{2}\left( x\right)
\succ \beta \left( y\vee _{(\alpha ,\beta )}z\right) +\alpha ^{2}\left(
x\right) \vee \beta \left( y\ast _{(\alpha ,\beta )}z\right)  \\
&=&\alpha \left( x\right) \succ _{(\alpha ,\beta )}\left( y\vee _{(\alpha
,\beta )}z\right) +\alpha \left( x\right) \vee _{(\alpha ,\beta )}\left(
y\ast _{(\alpha ,\beta )}z\right) .
\end{eqnarray*}%
Thus, we proved that $A_{(\alpha , \beta )}$ is a BiHom-NS-algebra.

The last statement is easy to prove and left to the reader.
\end{proof}

Similarly to the classical case (see \cite{uchinoarXiv}), BiHom-NS-algebras are actually a generalization of BiHom-tridendriform algebras, as shown by the following  result. 
\begin{proposition}\label{prop2.4}
Let $(A,\prec ,\succ ,\cdot ,\alpha ,\beta )$ be a BiHom-tridendriform
algebra. Then $(A,\prec ,\succ ,\cdot ,\alpha ,\beta )$ is a
BiHom-NS-algebra.
\end{proposition}
\begin{proof}
Note that the first 6 axioms defining a BiHom-tridendriform algebra ((\ref{BiHomtridend1})--(\ref{BiHomtridend10})) are
identical to the first 6 axioms defining a BiHom-NS-algebra ((\ref{BiHomNS0})--(\ref{BiHomNS5})), so we only have to check
the relation (\ref{BiHomNS6}). We compute: \\[2mm]
${\;\;\;\;\;}$
$\left( x\cdot y\right)  \prec \beta \left( z\right) +\left( x\ast y\right)
\cdot \beta \left( z\right)$
\begin{eqnarray*}
&=&\left( x\cdot y\right) \prec \beta \left( z\right) +\left( x\succ
y+x\prec y+x\cdot y\right) \cdot \beta \left( z\right)  \\
&\overset{\left( \ref{BiHomtridend13}\right) }{=}&\alpha (x)\cdot (y \prec
z)+\left( x\succ y\right) \cdot \beta \left( z\right) +\left( x\prec
y\right) \cdot \beta \left( z\right) +\left( x\cdot y\right) \cdot \beta
\left( z\right)  \\
&\overset{\left( \ref{BiHomtridend12}\right) ,\left( \ref{BiHomtridend11}%
\right) ,\left( \ref{BiHomtridend14}\right) }{=}&\alpha (x)\cdot (y \prec
z)+\alpha (x)\succ (y\cdot z)+\alpha (x)\cdot (y\succ z)+\alpha (x)\cdot
(y\cdot z) \\
&=&\alpha (x)\cdot \left( y\prec z+ y\succ z + y\cdot z \right) +\alpha
(x)\succ (y\cdot z) \\
&=&\alpha \left(
x\right) \cdot \left( y\ast z\right) + \alpha \left( x\right) \succ \left( y\cdot z\right),
\end{eqnarray*}
finishing the proof.
\end{proof}

Our aim now is to give a more conceptual proof of Proposition \ref{prop2.4}, inspired by \cite{opv}. Along the way we will introduce 
bimodule algebras over BiHom-associative algebras and use them to obtain a characterization of BiHom-tridendriform algebras. 

We begin by recalling the following observation from \cite{lmmp1}.
Let $(A, \mu _A, \alpha _A, \beta _A)$ be a BiHom-associative algebra, $M$ be a $\Bbbk $-linear space, $\alpha _M, \beta _M
:M\rightarrow M$  be two commuting linear maps and $l:A\otimes M\rightarrow M$, $x\otimes m\mapsto x\cdot m$ and $r:M\otimes A\rightarrow M$, $m\otimes x\mapsto m\cdot x$ be two linear maps. On the direct sum $A\oplus M$ consider the algebra structure 
defined by 
\begin{eqnarray*}
&(x, m)(x', m')=(xx', x\cdot m'+m\cdot x'), \;\;\; 
\forall \; x, x'\in A, \;m, m'\in M
\end{eqnarray*}
(the split null extension), denoted by $A\oplus _0M$. Then $A\oplus _0M$ is a 
BiHom-associative algebra with structure maps $\alpha , \beta :A\oplus _0M\rightarrow A\oplus _0M$, 
$\alpha ((x, m))=(\alpha _A(x), \alpha _M(m))$ and $\beta ((x, m))=(\beta _A(x), \beta _M(m))$, if and only if 
$(M, l, r, \alpha _M, \beta _M)$ is an $A$-bimodule. 

Assume again that $(A, \mu _A, \alpha _A, \beta _A)$ is a BiHom-associative algebra, $M$ is a $\Bbbk $-linear space, $\alpha _M, \beta _M
:M\rightarrow M$ are two commuting linear maps and $l:A\otimes M\rightarrow M$, $x\otimes m\mapsto x\cdot m$ and $r:M\otimes A\rightarrow M$, $m\otimes x\mapsto m\cdot x$ are two linear maps, but assume that we also have a linear map $\bullet 
:M\otimes M\rightarrow M$. 
\begin{definition}\label{bimalg}
We say that $(M, \bullet, l, r, \alpha _M, \beta _M)$ is an $A$-bimodule algebra if $(A\oplus M, *_{\bullet }, \alpha , \beta )$ 
is a BiHom-associative algebra, where 
$$(x, m)*_{\bullet }(x', m')=(xx', x\cdot m'+m\cdot x'+m\bullet m')$$
and $\alpha ((x, m))=(\alpha _A(x), \alpha _M(m))$, $\beta ((x, m))=(\beta _A(x), \beta _M(m))$, for all $x, x'\in A$ 
and $m, m'\in M$. 
\end{definition}

We show now that this concept is indeed the BiHom analogue of the corresponding concept introduced for associative algebras 
in \cite{pacific}. 
\begin{proposition}\label{bimalgbim}
With the above notations, $(M, \bullet, l, r, \alpha _M, \beta _M)$ is an $A$-bimodule algebra if and only if 
$(M, l, r, \alpha _M, \beta _M)$ is an $A$-bimodule, $(M, \bullet, \alpha _M, \beta _M)$ is a BiHom-associative algebra 
and the following relations hold, for all $x\in A$ and $m, m'\in M$:
\begin{eqnarray}
&&\alpha _A(x)\cdot (m\bullet m')=(x\cdot m)\bullet \beta _M(m'), \label{extra1} \\
&&\alpha _M(m)\bullet (m'\cdot x)=(m\bullet m')\cdot \beta _A(x), \label{extra2} \\
&&\alpha _M(m)\bullet (x\cdot m')=(m\cdot x)\bullet \beta _M(m'). \label{extra3}
\end{eqnarray}
In particular, $A$-bimodule algebra implies $A$-bimodule. 
\end{proposition}
\begin{proof}
We only sketch the proof and leave the details to the reader. The BiHom-associativity condition for 
$(A\oplus M, *_{\bullet }, \alpha , \beta )$ written for elements $(x, m), (x', m'), (x'', m'')\in A\oplus M$ turns out 
to be equivalent to 
\begin{eqnarray*}
&&\alpha _A(x)\cdot (x'\cdot m'')+\alpha _A(x)\cdot (m'\cdot x'')+\alpha _A(x)\cdot (m'\bullet m'')+
\alpha _M(m)\cdot (x'x'')\\
&&\;\;\;\;\;\;\;\;+\alpha _M(m)\bullet (x'\cdot m'')+\alpha _M(m)\bullet (m'\cdot x'')+\alpha _M(m)\bullet (m'\bullet m'')\\
&&=(xx')\cdot \beta _M(m'')+(x\cdot m')\cdot \beta _A(x'')+(m\cdot x')\cdot \beta _A(x'')+(m\bullet m')\cdot \beta _A(x'')\\
&&\;\;\;\;\;\;\;\;+(x\cdot m')\bullet \beta _M(m'')+(m\cdot x')\bullet \beta _M(m'')+(m\bullet m')\bullet \beta _M(m'').
\end{eqnarray*}
By taking $x=x'=x''=0$ we obtain $\alpha _M(m)\bullet (m'\bullet m'')=(m\bullet m')\bullet \beta _M(m'')$. By taking 
$x''=0$, $m=m'=0$, then $x=0$, $m'=m''=0$ and then $x'=0$, $m=m''=0$, we respectively obtain 
$\alpha _A(x)\cdot (x'\cdot m'')=(xx')\cdot \beta _M(m'')$, $\alpha _M(m)\cdot (x'x'')=(m\cdot x')\cdot \beta _A(x'')$ and 
$\alpha _A(x)\cdot (m'\cdot x'')=(x\cdot m')\cdot \beta _A(x'')$, so $(M, l, r, \alpha _M, \beta _M)$ is an $A$-bimodule. 
Similarly one obtains the relations (\ref{extra1}), (\ref{extra2}) and (\ref{extra3}). 

The converse is immediate. 
\end{proof}

Similarly to the characterization of BiHom-dendriform algebras in \cite{lmmp1}, Proposition 3.16 and the characterization of BiHom-NS-algebras in Proposition \ref{BHNStoBHA}, we can characterize now BiHom-tridendriform algebras.
\begin{proposition}\label{characttridend}
Let $\left( A,\prec ,\succ ,\cdot ,\alpha ,\beta \right) $ be a 6-tuple where $A$ is a 
$\Bbbk$-linear space and $\prec
,\succ ,\cdot :A\otimes A\rightarrow A$ and $\alpha, \beta :A\rightarrow A$ are linear maps such that $\alpha $ and $\beta $ are 
multiplicative with respect to $\prec $, $\succ $ and $\cdot $. Define a new multiplication on $A$ by 
$x\ast y=x\prec y+x\succ y+x\cdot y$, for all $x, y\in A$. Then 
$\left( A,\prec ,\succ ,\cdot , \alpha , \beta \right) $ is a  BiHom-tridendriform algebra if and only if 
$\left( A,\ast ,\alpha ,\beta \right) $ is a BiHom-associative algebra and 
$(A, \cdot , \succ , \prec, \alpha , \beta )$ is an $\left( A,\ast ,\alpha ,\beta \right) $-bimodule algebra 
(notation for bimodule algebras as in Definition \ref{bimalg}). 
\end{proposition}
\begin{proof}
Assume that $\left( A,\prec ,\succ ,\cdot , \alpha , \beta \right) $ is BiHom-tridendriform. The tuple 
$(A, \ast, \alpha , \beta )$ defines a  BiHom-associative algebra by \cite{lmmp1}, Proposition 3.14, so we only have to prove that 
$(A, \cdot , \succ , \prec, \alpha , \beta )$ is an $\left( A,\ast ,\alpha ,\beta \right) $-bimodule algebra. The fact that 
$(A, \cdot , \alpha , \beta )$ is BiHom-associative is equivalent to (\ref{BiHomtridend14}), while each of  conditions 
(\ref{lmod}), (\ref{rmod}), (\ref{bimod}), (\ref{extra1}),  (\ref{extra2}), (\ref{extra3}) is equivalent to one of the axioms 
(\ref{BiHomtridend8})-(\ref{BiHomtridend13}). These observations indicate also that the converse holds as well. 
\end{proof}

Thus, as a consequence of Proposition \ref{BHNStoBHA}, Proposition \ref{characttridend} and the fact from Proposition 
\ref{bimalgbim} that bimodule algebra implies bimodule, we reobtain Proposition \ref{prop2.4}.
\section{Twisted Rota-Baxter operators}\label{sec2}
\setcounter{equation}{0}
We introduce in this section the concept of twisted Rota-Baxter operator in the context of BiHom-associative algebras and show that there is an adjunction with BiHom-NS-algebras.

Let $(A, \mu _A, \alpha _A, \beta _A)$ be a BiHom-associative algebra and $(M, l, r, \alpha _M, \beta _M)$ an $A$-bimodule, 
with notation as in Definition \ref{bimodule}.  
Hochschild cohomology of $A$ with coefficients in $M$ has been introduced in \cite{das}. In particular, a 
Hochschild 2-cocycle on $A$ with values in $M$ 
 is a linear map $H:A\otimes A\rightarrow M$ satisfying  
(for all $x, y, z\in A$): 
\begin{eqnarray}
&& H\circ (\alpha_{A}\otimes \alpha_{A})=\alpha_{M}\circ H, \;\;\;\;\;H\circ (\beta_{A}\otimes \beta_{A})=\beta_{M}\circ H, 
\label{Hoc1}\\
&&\alpha_{A}(x)\cdot H(y, z)- H(xy, \beta_A(z))+ H(\alpha_{A}(x), yz)- H(x, y)\cdot \beta_{A}(z)=0. \label{cocycle}
\end{eqnarray}

Inspired by \cite{uchino}, we introduce the following concept:
\begin{definition}\label{DefTRB}
Let $(A, \mu _A, \alpha_{A} , \beta_{A} )$ be a BiHom-associative algebra, let $(M, l, r, \alpha_{M}, \beta_{M})$ be an $A$-bimodule 
and let $H:A\otimes A\rightarrow M$ be a Hochschild 2-cocycle. 
A $H$-twisted Rota-Baxter operator is a linear map $\pi:
M\rightarrow A$ satisfying the following conditions, for all $m, n\in M$: 
\begin{eqnarray}
&&\pi\circ \alpha_{M} =\alpha_{A} \circ \pi, \;\;\;\;\; \pi\circ \beta_{M} =\beta_{A} \circ \pi , \label{supTRB} \\
&&\pi(m) \pi(n)=\pi(\pi(m)\cdot n+ m\cdot \pi(n)+ H(\pi(m), \pi(n))).  \label{trb}
\end{eqnarray}
\end{definition}
\begin{example}
Let $(A, \mu, \alpha , \beta )$ be a BiHom-associative algebra. One can easily see that $(-\mu ):A\otimes A\rightarrow A$ is a 
Hochschild 2-cocycle on $A$ with values in the bimodule $(A, l, r, \alpha , \beta )$, where $l(x\otimes y)=r(x\otimes y)=xy$, 
for all $x, y\in A$. By analogy with the classical case (see \cite{rota}, \cite{uchino}), a $(-\mu )$-twisted Rota-Baxter 
operator is called a Reynolds operator. Thus, a Reynolds operator on $A$ is a linear map $R:A\rightarrow A$ commuting 
with $\alpha $ and $\beta $ and satisfying the relation 
\begin{eqnarray*}
&&R(x)R(y)=R(R(x)y+xR(y)-R(x)R(y)), \;\;\; \forall \; x, y\in A. 
\end{eqnarray*}
\end{example}

We prove now that twisted Rota-Baxter operators provide examples of BiHom-NS-algebras, extending the situation 
in the associative case in \cite{uchino}, Proposition 3.8. 
\begin{proposition} \label{TRB}
Let $(A, \mu _A, \alpha_{A} , \beta_{A} )$ be a BiHom-associative algebra, let $(M, l, r, \alpha_{M}, \beta_{M})$ be an 
$A$-bimodule, let $H:A\otimes A\rightarrow M$ be a Hochschild 2-cocycle and let 
$\pi:M\rightarrow A$ be a $H$-twisted Rota-Baxter operator. Define the following operations on $A$:
\begin{eqnarray*}
&&m\prec n=m\cdot \pi(n), \;\;\;\;\;
m\succ n=\pi(m)\cdot n, \;\;\;\;\; m\vee n=H(\pi(m), \pi(n)), 
\end{eqnarray*}
for all $m, n\in M$. 
Then $(M, \prec, \succ, \vee, \alpha_{M}, \beta_{M})$ is a BiHom-NS-algebra. 
\end{proposition}
\begin{proof}
We need to check the relations (\ref{BiHomNS0})-(\ref{BiHomNS6}). By Definition \ref{bimodule}, $\alpha _M$ and 
$\beta _M$ commute, so (\ref{BiHomNS0}) holds. The relations  (\ref{BiHomNS1}) and (\ref{BiHomNS2}) are easy to prove, 
by using the bimodule axioms and (\ref{Hoc1}), (\ref{supTRB}). Let now $x, y, z\in M$;  
to prove (\ref{BiHomNS3}), (\ref{BiHomNS4}) and (\ref{BiHomNS5}), we compute:
\begin{eqnarray*}
(x\prec y)\prec \beta_{M}(z)&=& (x\cdot \pi(y))\cdot \pi(\beta_{M}(z))
 \overset{\left( \ref{supTRB}\right) }{=}(x\cdot \pi(y))\cdot \beta_{A}(\pi(z))\\
&\overset{(\ref{rmod})}{=}& \alpha_{M}(x)\cdot (\pi(y)\pi(z))
\overset{(\ref{trb})}{=} \alpha_{M}(x)\cdot \pi(\pi(y)\cdot z + y\cdot \pi(z)+ H(\pi(y), \pi(z)))\\
&=& \alpha_{M}(x)\cdot \pi(y\succ z + y\prec z + y\vee z)
=\alpha_{M}(x)\prec (y\ast z), 
\end{eqnarray*}
\begin{eqnarray*}
(x\succ y)\prec \beta_{M}(z)&=& (\pi(x)\cdot y)\cdot \pi(\beta_{M}(z))
\overset{\left( \ref{supTRB}\right) }{=}(\pi(x)\cdot y)\cdot \beta_{A}(\pi(z))\\
&\overset{(\ref{bimod})}{=}& \alpha_{A}(\pi(x))\cdot (y\cdot \pi(z))
\overset{\left( \ref{supTRB}\right) }{=}\pi(\alpha_{M}(x))\cdot (y\cdot \pi(z))
=\alpha_{M}(x)\succ (y\prec z), 
\end{eqnarray*}
\begin{eqnarray*}
(x\ast y)\succ \beta_{M}(z)&=& (\pi(x)\cdot y + x\cdot \pi(y)+ H(\pi(x), \pi(y)))\succ \beta_{M}(z)\\
&=& \pi(\pi(x)\cdot y + x\cdot \pi(y)+ H(\pi(x), \pi(y)))\cdot \beta_{M}(z)
\overset{(\ref{trb})}{=}(\pi(x)\pi(y))\cdot \beta_{M}(z)\\
&\overset{(\ref{lmod})}{=}& \alpha_{A}(\pi(x))\cdot (\pi(y)\cdot z)
\overset{\left( \ref{supTRB}\right) }{=}\pi(\alpha_{M}(x))\cdot (\pi(y)\cdot z)
=\alpha_{M}(x)\succ (y\succ z).
\end{eqnarray*}
Finally, we check (\ref{BiHomNS6}):\\[2mm]
${\;\;\;}$
$\alpha_{M}(x)\succ (y\vee z)-(x\ast y)\vee \beta_{M}(z)+\alpha_{M}(x)\vee (y\ast z)-(x\vee y)\prec \beta_{M}(z)$
\begin{eqnarray*}
&=& \pi(\alpha_{M}(x))\cdot H(\pi(y), \pi(z))-(x\succ y + x\prec y + x\vee y)\vee \beta_{M}(z)\\
&& +\alpha_{M}(x)\vee (y\succ z + y\prec z + y\vee z)- H(\pi(x), \pi(y))\prec \beta_{M}(z)\\
&=& \pi(\alpha_{M}(x))\cdot H(\pi(y), \pi(z))-H(\pi(\pi(x)\cdot y + x\cdot \pi(y) + H(\pi(x), \pi(y))), \pi(\beta_{M}(z)))\\
&& +H(\pi(\alpha_{M}(x)), \pi(\pi(y)\cdot z + y\cdot \pi(z) + H(\pi(y), \pi(z))))- H(\pi(x), \pi(y))\cdot \pi(\beta_{M}(z))\\
&\overset{(\ref{trb})}{=}& \pi(\alpha_{M}(x))\cdot H(\pi(y), \pi(z))-H(\pi(x)\pi(y), \pi(\beta_{M}(z))\\
&& +H(\pi(\alpha_{M}(x)), \pi(y)\pi(z))
- H(\pi(x), \pi(y))\cdot \pi(\beta_{M}(z))\\
&\overset{\left( \ref{supTRB}\right) }{=}& \alpha_{A}(\pi(x))\cdot H(\pi(y), \pi(z))-H(\pi(x)\pi(y), \beta_{A}(\pi(z))\\
&&+H(\alpha_{A}(\pi(x)), \pi(y)\pi(z))
- H(\pi(x), \pi(y))\cdot \beta_{A}(\pi(z))
\overset{(\ref{cocycle})}{=}0,
\end{eqnarray*}
finishing the proof. 
\end{proof}
 \begin{lemma}\label{2Cocycle}
Let $(A,\prec ,\succ ,\vee,\alpha,\beta )$ be a BiHom-NS-algebra and consider the BiHom-associative algebra  $A_{bhas}$  
and the $A_{bhas}$-bimodule $(A, \succ , \prec , \alpha , \beta )$ as in Proposition \ref{BHNStoBHA}. Define the linear map 
$H:A\otimes A\rightarrow A$, $H (x,y) = x\vee y$, for all $x, y\in A$. 
Then $H$ is a Hochschild 2-cocycle on $A_{bhas}$ with values in $(A, \succ , \prec , \alpha , \beta )$. 
 \end{lemma}
 \begin{proof}
The relation (\ref{Hoc1}) for $H$ follows immediately from the the multiplicativity of $\alpha $ and $\beta $ with respect to 
the operation $\vee$. We need to check the  cocycle condition \eqref{cocycle}.
We have: \\[2mm]
${\;\;\;\;\;}$
$\alpha(x)\succ  H(y, z)- H(x\ast y, \beta (z))+ H (\alpha(x), y\ast z)- H(x, y)\prec \beta(z)$
\begin{eqnarray*}
&=& \alpha(x)\succ  (y \vee z)- (x\ast y)\vee \beta (z)+ \alpha(x)\vee (y\ast z)-  (x\vee y)\prec \beta(z)
\overset{(\ref{BiHomNS6})}{=}0,
\end{eqnarray*}
finishing the proof.
\end{proof}

Let $\bf{BHNS}$ denote the category of BiHom-NS-algebras, where the objects are BiHom-NS-algebras and the morphisms are as described  in Definition \ref{DefBHNS}. Let $\bf{TRB}$ denote the category of twisted Rota-Baxter operators, where the objects are the twisted Rota-Baxter operators described  in Definition \ref{DefTRB} and morphisms are defined by the commutative diagram

 $$
\xymatrix{
  A\otimes A \ar[d]_{H} \ar[r]^{\varphi\otimes \varphi}
                & A'\otimes A' \ar[d]^{H'}  \\
                M \ar[d]_{\pi} \ar[r]^{\psi}
                & M' \ar[d]^{\theta}  \\
    A \ar[r]_{\varphi}
                & A'\\           }   
                $$
and compatibility with linear maps.

\begin{theorem}
There is an adjoint pair of functors
   \begin{equation}
  F \colon \bf{BHNS} \rightleftarrows {\bf{TRB}} \colon G ,
   \end{equation}
   with the adjoint relation (for $(\pi , H)$ an object in $\bf{TRB}$)
$$ Hom_{\bf{TRB}}(F(A),(\pi,H))\cong Hom_{\bf{BHNS}}(A,G(\pi,H)).$$
\end{theorem}
\begin{proof}
The functor $G$ is given by Proposition \ref{TRB} and the functor $F$ is defined by the identity map 
$id_A:(A, \succ, \prec, \alpha , \beta )\rightarrow A_{bhas}$ as a twisted Rota-Baxter operator with a 2-cocycle defined as $H(x,y)=x\vee y$ for all $x,y\in A$ (Lemma \ref{2Cocycle}).\end{proof}

\section{Generalized Nijenhuis operators}\label{sec3}
\setcounter{equation}{0}

On a BiHom-associative algebra, one can consider usual Rota-Baxter operators (as in \cite{lmmp1}), but also a more general version, 
depending on some extra maps, called $\{\sigma , \tau \}$-Rota-Baxter operators in \cite{canadian}. The aim of this section 
is to show that a similar situation occurs for Nijenhuis operators. 

We begin by introducing a sort of Nijenhuis analogue of $\{\sigma , \tau \}$-Rota-Baxter operators, and prove that they 
satisfy the expected property, namely they lead to BiHom-NS-algebras. 
\begin{theorem} \label{mainnij}
Let $(A, \mu , \alpha , \beta )$ be a BiHom-associative algebra, let $\sigma, \gamma, \tau, \delta :A\rightarrow A$ be linear maps that are multiplicative with respect to $\mu$ and such that any two of the maps $\alpha , \beta , \sigma, \gamma, \tau, \delta $
commute and let $N:A\rightarrow A$ be a linear map satisfying the following conditions (for all $x, y\in A$):
\begin{eqnarray}
&&\alpha \sigma \gamma N=N\alpha \sigma \gamma , \label{extracom1}\\
&&\beta \tau \delta N=N\beta \tau \delta , \label{extracom2} \\
&&\sigma\gamma N(x)\tau\delta N(y)=N(\sigma \gamma (x)\delta N(y)+\gamma N(x)\tau \delta (y)-N(\gamma (x)\delta (y))), 
\label{genNij} \\
&&\alpha \sigma\gamma ^2N(x)\tau\delta N(y)=N(\alpha \sigma \gamma ^2(x)\delta N(y)+\alpha \gamma ^2N(x)\tau \delta (y)-N(\alpha \gamma ^2(x)\delta (y))), 
\label{genNijsup1} \\
&&\sigma\gamma N(x)\beta \tau\delta ^2N(y)=N(\sigma \gamma (x)\beta \delta ^2N(y)+\gamma N(x)\beta \tau \delta ^2(y)-N(\gamma (x)\beta \delta ^2(y)))
\label{genNijsup2} 
\end{eqnarray}
(we call such  $N$ a generalized Nijenhuis operator). Define the following operations on $A$: 
\begin{eqnarray*}
&&x\prec y=\sigma \gamma (x)\delta N(y), \;\;\;\;\;x\succ y=\gamma N(x)\tau \delta (y), \;\;\;\;\;
x\vee y=-N(\gamma (x)\delta (y)),
\end{eqnarray*}
for all $x, y\in A$. Then $(A, \prec, \succ, \vee, \alpha \sigma \gamma , \beta \tau \delta )$ is a BiHom-NS-algebra. 
In particular, if we define a new multiplication on $A$ by $x\ast y=\sigma \gamma (x)\delta N(y)+\gamma N(x)\tau \delta (y)-
N(\gamma (x)\delta (y))$, then $(A, \ast, \alpha \sigma \gamma , \beta \tau \delta )$ is a BiHom-associative algebra. 
\end{theorem}
\begin{proof}
We need to check the conditions (\ref{BiHomNS0})-(\ref{BiHomNS6}). The condition (\ref{BiHomNS0}) is 
obviously satisfied, while (\ref{BiHomNS1}) and (\ref{BiHomNS2}) follow immediately by using (\ref{extracom1}) and 
(\ref{extracom2}). To check  (\ref{BiHomNS3}), (\ref{BiHomNS4}) and (\ref{BiHomNS5}) we compute: 
\begin{eqnarray*}
(x\prec y)\prec \beta \tau \delta (z)
&=&(\sigma \gamma (x)\delta N(y))\prec \beta \tau \delta (z)=(\sigma ^2\gamma ^2(x)\sigma \gamma \delta N(y))\delta N\beta \tau \delta (z)\\
&\overset{(\ref{extracom2})}{=}&(\sigma ^2\gamma ^2(x)\sigma \gamma \delta N(y))\beta \tau \delta ^2N(z)\\
&\overset{(\ref{eqasso})}{=}&\alpha \sigma ^2\gamma ^2(x)(\sigma \gamma \delta N(y) \tau \delta ^2N(z))=
\alpha \sigma ^2\gamma ^2(x)\delta (\sigma \gamma N(y) \tau \delta N(z))\\
&\overset{(\ref{genNij})}{=}&\alpha \sigma ^2\gamma ^2(x)(\delta N(\sigma \gamma (y)\delta N(z)+\gamma N(y)\tau \delta (z)-N(\gamma (y)\delta (z))))\\
&=&\alpha \sigma ^2\gamma ^2(x)\delta N(y\ast z) =\sigma \gamma (\alpha \sigma \gamma (x))\delta N(y\ast z)
=\alpha \sigma \gamma (x)\prec (y\ast z), 
\end{eqnarray*}
\begin{eqnarray*}
(x\succ y)\prec \beta \tau \delta (z)&=&(\gamma N(x)\tau \delta (y))\prec \beta \tau \delta (z)=
(\sigma \gamma ^2N(x)\sigma \gamma \tau \delta (y))\delta N\beta \tau \delta (z)\\
&\overset{(\ref{extracom2})}{=}&(\sigma \gamma ^2N(x)\sigma \gamma \tau \delta (y))\beta \tau \delta ^2N(z)
\overset{(\ref{eqasso})}{=}\alpha \sigma \gamma ^2N(x)(\sigma \gamma \tau \delta (y)\tau \delta ^2N(z))\\
&=&\alpha \sigma \gamma ^2N(x)\tau \delta (\sigma \gamma (y)\delta N(z))=\alpha \sigma \gamma ^2N(x)\tau \delta 
(y\prec z)\\
&\overset{(\ref{extracom1})}{=}&\gamma N(\alpha \sigma \gamma (x))\tau \delta (y\prec z)=\alpha \sigma \gamma (x)\succ 
(y\prec z), 
\end{eqnarray*}
\begin{eqnarray*}
\alpha \sigma \gamma (x)\succ (y\succ z)&=&\alpha \sigma \gamma (x)\succ (\gamma N(y)\tau \delta (z))
=\gamma N\alpha \sigma \gamma (x)(\tau \delta \gamma N(y)\tau ^2\delta ^2(z))\\
&\overset{(\ref{extracom1})}{=}&\alpha \sigma \gamma ^2N(x)(\tau \delta \gamma N(y)\tau ^2\delta ^2(z))
\overset{(\ref{eqasso})}{=}(\sigma \gamma ^2N(x)\tau \delta \gamma N(y))\beta \tau ^2\delta ^2(z)\\
&=&\gamma (\sigma \gamma N(x)\tau \delta N(y))\beta \tau ^2\delta ^2(z)\overset{(\ref{genNij})}{=}
\gamma N(x\ast y)\beta \tau ^2\delta ^2(z)\\
&=&\gamma N(x\ast y)\tau \delta (\beta \tau \delta (z))=(x\ast y)\succ \beta \tau \delta (z). 
\end{eqnarray*}
Finally, we check (\ref{BiHomNS6}):\\[2mm]
${\;\;\;\;\;}$
$(x\vee y)\prec \beta \tau \delta (z)+(x\ast y)\vee \beta \tau \delta (z)$
\begin{eqnarray*}
&=&-N(\gamma (x)\delta (y))\prec \beta \tau \delta (z)+(\sigma \gamma (x)\delta N(y))\vee \beta \tau \delta (z)\\
&&+(\gamma N(x)\tau \delta (y))\vee \beta \tau \delta (z)-N(\gamma (x)\delta (y))\vee \beta \tau \delta (z)\\
&=&-\sigma \gamma N(\gamma (x)\delta (y))\delta N\beta \tau \delta (z)-N((\sigma \gamma ^2(x)\gamma \delta N(y))
\delta \beta \tau \delta (z))\\
&&-N((\gamma ^2N(x)\tau \delta \gamma (y))\delta \beta \tau \delta (z))+N(\gamma N(\gamma (x)\delta (y))
\delta \beta \tau \delta (z))\\
&\overset{(\ref{extracom2}), (\ref{eqasso})}{=}&-\sigma \gamma N(\gamma (x)\delta (y))\beta \tau \delta ^2N(z)-
N(\alpha \sigma \gamma ^2(x)(\gamma \delta N(y)\tau \delta ^2(z)))\\
&&-N(\alpha \gamma ^2N(x)(\tau \delta \gamma (y)\tau \delta ^2(z)))+
N(\gamma N(\gamma (x)\delta (y))\delta \beta \tau \delta (z))\\
&\overset{(\ref{genNijsup2})}{=}&-N((\sigma \gamma ^2(x)\sigma \gamma \delta (y))\beta \delta ^2N(z))-
N(\gamma N(\gamma (x)\delta (y))\beta \tau \delta ^2(z))\\
&&+N^2((\gamma ^2(x)\gamma \delta (y))\beta \delta ^2(z))-N(\alpha \sigma \gamma ^2(x)(\gamma \delta N(y)
\tau \delta ^2(z)))\\
&&-N(\alpha \gamma ^2N(x)(\tau \delta \gamma (y)\tau \delta ^2(z)))+
N(\gamma N(\gamma (x)\delta (y))\delta \beta \tau \delta (z))\\
&\overset{(\ref{eqasso})}{=}&-N(\alpha \sigma \gamma ^2(x)(\sigma \gamma \delta (y)\delta ^2N(z)))
+N^2(\alpha \gamma ^2(x)\delta (\gamma (y)\delta (z)))\\
&&-N(\alpha \sigma \gamma ^2(x)(\gamma \delta N(y)\tau \delta ^2(z)))-
N(\alpha \gamma ^2N(x)\tau \delta (\gamma (y)\delta (z)))\\
&\overset{(\ref{genNijsup1})}{=}&-N(\alpha \sigma \gamma ^2(x)(\sigma \gamma \delta (y)\delta ^2N(z)))-
N(\alpha \sigma \gamma ^2(x)(\gamma \delta N(y)\tau \delta ^2(z)))\\
&&+N(\alpha \sigma \gamma ^2(x)\delta N(\gamma (y)\delta (z)))-
\alpha \sigma \gamma ^2N(x)\tau \delta N(\gamma (y)\delta (z))\\
&=&N(\alpha \sigma \gamma ^2(x)\delta (N(\gamma (y)\delta (z))-\sigma \gamma (y)\delta N(z)-\gamma N(y)\tau \delta (z)))\\
&&-
\alpha \sigma \gamma ^2N(x)\tau \delta N(\gamma (y)\delta (z))\\
&=&-N(\alpha \sigma \gamma ^2(x)\delta (y\ast z))+\alpha \sigma \gamma ^2N(x)\tau \delta (y\vee z)\\
&\overset{(\ref{extracom1})}{=}&-N(\gamma (\alpha \sigma \gamma (x))\delta (y\ast z))
+\gamma N(\alpha \sigma \gamma (x))\tau \delta (y\vee z)\\
&=&\alpha \sigma \gamma (x)\vee (y\ast z)+\alpha \sigma \gamma (x)\succ (y\vee z), 
\end{eqnarray*}
finishing the proof. 
\end{proof}

We consider now some particular cases of Theorem \ref{mainnij}.
\begin{corollary}\label{firstconseq}
Let $(A, \mu , \alpha , \beta )$ be a BiHom-associative algebra and let $N:A\rightarrow A$ be a linear map commuting with 
$\alpha ^2$ and $\beta ^2$ and 
satisfying the following conditions (for all $x, y\in A$):
\begin{eqnarray}
&&\alpha N(x)\beta N(y)=N(\alpha (x)N(y)+\alpha N(x)\beta (y)-N(\alpha (x)y)), 
\label{gN1} \\
&&\alpha N(x)\beta ^2N(y)=N(\alpha (x)\beta N(y)+\alpha N(x)\beta ^2(y)-N(\alpha (x)\beta (y))).
\label{gN2}
\end{eqnarray}
Define new operations on $A$ by  
$x\prec y=\alpha (x)N(y)$, $x\succ y=\alpha N(x)\beta (y)$, 
$x\vee y=-N(\alpha (x)y)$, 
for all $x, y\in A$. Then $(A, \prec, \succ, \vee, \alpha ^2 , \beta ^2 )$ is a BiHom-NS-algebra. 
In particular, if we define a new multiplication on $A$ by $x\ast y=\alpha (x)N(y)+\alpha N(x)\beta (y)
-N(\alpha (x)y)$, then $(A, \ast, \alpha ^2 , \beta ^2 )$ is a BiHom-associative algebra. 
\end{corollary}
\begin{proof}
Take in Theorem \ref{mainnij} $\sigma =id$, $\gamma =\alpha $, $\tau =\beta $, $\delta =id$. Then (\ref{genNij}) becomes 
(\ref{gN1}), (\ref{genNijsup1}) is a consequence of (\ref{gN1}) and 
(\ref{genNijsup2}) becomes (\ref{gN2}). 
\end{proof}
\begin{corollary}\label{secondconseq}
Let $(A, \mu , \alpha , \beta )$ be a BiHom-associative algebra and let $N:A\rightarrow A$ be a linear map commuting with 
$\alpha ^2$ and $\beta ^2$ and 
satisfying the following conditions (for all $x, y\in A$):
\begin{eqnarray}
&&\alpha N(x)\beta N(y)=N(\alpha (x)\beta N(y)+ N(x)\beta (y)-N(x\beta (y))), 
\label{gN3} \\
&&\alpha ^2N(x)\beta N(y)=N(\alpha ^2(x)\beta N(y)+ \alpha N(x)\beta (y)-N(\alpha (x)\beta (y))). 
\label{gN4}
\end{eqnarray}
Define new operations on $A$ by $x\prec y=\alpha (x)\beta N(y)$, $x\succ y=N(x)\beta (y)$, $x\vee y=-N(x\beta (y))$,  
for all $x, y\in A$. Then $(A, \prec, \succ, \vee, \alpha ^2, \beta ^2 )$ is a BiHom-NS-algebra. 
In particular, if we define a new multiplication on $A$ by $x\ast y=\alpha (x)\beta N(y)+N(x)\beta (y)-N(x\beta (y))$, 
then $(A, \ast, \alpha ^2 , \beta ^2 )$ is a BiHom-associative algebra. 
\end{corollary}
\begin{proof}
Take in Theorem \ref{mainnij} $\sigma =\alpha $, $\gamma =id $, $\tau =id $, $\delta =\beta$. Then (\ref{genNij}) becomes 
(\ref{gN3}), (\ref{genNijsup1}) becomes (\ref{gN4}) and (\ref{genNijsup2}) is a consequence of (\ref{gN3}). 
\end{proof}
\begin{remark}
Note that, in Corollary \ref{firstconseq}, if $N$ commutes with $\beta $, then (\ref{gN2}) is a consequence of (\ref{gN1}); 
however, in general $N$ does not commute with $\beta $, such a situation occurs for instance in Example \ref{expjgeomphys} 
below. 
A similar discussion holds for Corollary 
\ref{secondconseq}. 
\end{remark}

Finally, by taking in Theorem \ref{mainnij} $\sigma =\gamma =\tau =\delta =id$, one can easily see that we obtain:
\begin{corollary}
Let $(A, \mu , \alpha , \beta )$ be a BiHom-associative algebra and let $N:A\rightarrow A$ be a Nijenhuis operator in the 
usual sense, that it 
\begin{eqnarray}
&&N(x)N(y)=N(xN(y)+N(x)y-N(xy)), \;\;\; \forall \; x, y\in A, \label{Nijeq}
\end{eqnarray}
which commutes with $\alpha $ and $\beta $. If we define new operations on $A$ by $x\prec y=xN(y)$, 
$x\succ y=N(x)y$, $x\vee y=-N(xy)$, then  $(A, \prec, \succ, \vee, \alpha , \beta  )$ is a BiHom-NS-algebra. 
\end{corollary}
\begin{example}\label{expjgeomphys}
Let $(A, \mu , \alpha , \beta )$ be a BiHom-associative algebra and let $a\in A$ such that $\alpha ^2(a)=\beta ^2(a)=a$. 
Define $N_1, N_2:A\rightarrow A$ by $N_1(x)=\alpha (a)x$ and $N_2(x)=x\alpha (a)$, for all $x\in A$. Then one can easily 
see that $N_1$ satisfies the hypotheses of Corollary \ref{firstconseq} and $N_2$ satisfies the hypotheses of 
Corollary \ref{secondconseq}. In both cases, the new multiplication $\ast $ is given by the same formula, namely 
$x\ast y=\alpha (x)(\alpha (a)y)$, so it turns out that $(A, \ast , \alpha ^2, \beta ^2)$ is a BiHom-associative algebra, 
which was the content of  Lemma 3.1 in \cite{usjgeomphys}, for which we have thus obtained a Nijenhuis operator interpretation. 
\end{example}

\begin{center}
ACKNOWLEDGEMENTS
\end{center}
Ling Liu was supported by the NSF of China (No. 12071441). Abdenacer Makhlouf was partially supported by GDRI Eco-Math. 
This paper was written while Claudia Menini was a member of the
"National Group for Algebraic and Geometric Structures and their
Applications" (GNSAGA-INdAM) and was partially supported by MIUR within the
National Research Project PRIN 2017. Florin Panaite was partially supported by a grant from UEFISCDI,
project number PN-III-P4-PCE-2021-0282.

\end{document}